\documentclass[a4paper,12pt,reqno]{amsart}

\usepackage{amsmath}
\usepackage{amsfonts}
\usepackage{amssymb}
\usepackage{amsthm}
\numberwithin{equation}{section}

\newtheorem{theorem}[equation]{Theorem}
\newtheorem{lemma}[equation]{Lemma}
\newtheorem{proposition}[equation]{Proposition}
\newtheorem{corollary}[equation]{Corollary}

\theoremstyle{definition}
\newtheorem{definition}[equation]{Definition}

\theoremstyle{remark}
\newtheorem{remark}[equation]{Remark}

\def\kint_#1{\mathchoice%
          {\mathop{\kern 0.2em\vrule width 0.6em height 0.69678ex depth -0.58065ex
                  \kern -0.8em \intop}\nolimits_{\kern -0.4em#1}}%
          {\mathop{\kern 0.1em\vrule width 0.5em height 0.69678ex depth -0.60387ex
                  \kern -0.6em \intop}\nolimits_{#1}}%
          {\mathop{\kern 0.1em\vrule width 0.5em height 0.69678ex depth -0.60387ex
                  \kern -0.6em \intop}\nolimits_{#1}}%
          {\mathop{\kern 0.1em\vrule width 0.5em height 0.69678ex depth -0.60387ex
                  \kern -0.6em \intop}\nolimits_{#1}}}

\newcommand{\R}{\mathbb{R}}

\newcommand{\N}{\mathbb{N}}

\newcommand{\loc}{\operatorname{loc}}

\providecommand{\ch}[1]{\text{\raise 2pt \hbox{$\chi$}\kern-0.2pt}_{#1}}
\newcommand\dist{\operatorname{dist}}
\newcommand{\barint}{\mathop{\hbox{\vrule height3pt depth-2.7pt width.65em}\hskip-1em \int}}

\newcommand\Om{\Omega}
\newcommand\ti{\widetilde}
\newcommand\pip{\varphi}
\newcommand\eps{\epsilon}

\newcommand{\mybibitem}[7]{\bibitem[#1]{#1} #2. {\it #3,} #4 {\bf #5}
#6 #7.}

\begin{document}

\title[Regularity of quasiminimal sets]
{Regularity of sets with quasiminimal boundary surfaces in metric spaces}
\author[Kinnunen, Korte, Lorent and Shanmugalingam]{Juha Kinnunen, Riikka Korte, Andrew Lorent, and Nageswari Shanmugalingam}
\begin{abstract}
This paper studies regularity of perimiter quasiminimizing sets in metric measure spaces with
a doubling  measure and a Poincar\'e inequality. The main result shows that  the measure 
theoretic boundary of a quasiminimizing set coincides with
the topological boundary. 
We also show that such a set has finite Minkowski content and
 apply the regularity theory study rectifiability issues related to quasiminimal sets in strong $A_\infty$-weighted
Euclidean case.
\end{abstract}


\maketitle


\section{Introduction}

It is now a well-known fact that Euclidean sets with (locally) minimal surfaces have smooth boundary apart from
a set of co-dimension $2$. This result is due to De Giorgi, see \cite{DG1} and \cite{DG2}. The analogous result for Euclidean quasiminimal
surfaces is due to David and Semmes~\cite{DS1}, who showed that bounded sets with quasiminimal boundary surfaces are uniformly
rectifiable and are locally John domains. 

The paper~\cite{DS1} considered a double obstacle problem in constructing quasiminimal surfaces in Euclidean spaces;
A similar problem was considered by Caffarelli and de la Llave in~\cite{CL}, where the setting is $C^2$ Riemannian manifolds.
In~\cite[Theorem~1.1]{CL} it is shown that given an Euclidean hyperplane (and the manifold is obtained by a perturbation of the Euclidean
metric in a $C^2$-fashion) there is a quasiminimal surface in the Riemannian metric that lies close to the hyperplane.
In~\cite{KKST2} a double obstacle problem similar to the one considered by~\cite{DS1} was studied in the setting of doubling 
metric measure spaces supporting a $(1,1)$-Poincar\'e inequality. It is therefore natural to ask what type of regularity properties
do the minimizing sets have away from the boundaries of the obstacles.

In this paper we study the regularity properties of quasiminimal sets or, more precisely, quasiminimal boundary surfaces
in the setting of metric measure spaces with a doubling measure that supports a $(1,1)$-Poincar\'e inequality. We will show,
by modifying De Giorgi's technique using a part of the argument of David and Semmes, that
such a set is porous and satisfies a measure density property. 
In particular, this implies that the measure theoretic boundary of a quasiminimizing set coincides with
the topological boundary. We also show that such a set has finite Minkowski content.  
In the metric setting the classical definition of rectifiability may not be as widely applicable.  
For instance, in the  
setting of Heisenberg groups there are sets of finite perimeter that are not rectifiable~\cite{Mag}.  
Hence the finiteness of  
the Minkowski content is the best one can hope for in this generality. 
Since the problem studied in~\cite{CL} is a minimization problem and
comes with an associated PDE, the techniques used there are essentially of PDE. The problem studied in~\cite{DS1} is a 
quasiminimization problem, and hence we find some of the methods used in this paper to be more easily adaptable to the 
general metric measure space setting.  

In the last two sections of this paper we apply the regularity theory developed
in the first part of the paper to study rectifiability issues related to quasiminimal sets in strong $A_\infty$-weighted
Euclidean setting. Observe that when equipped with a strong $A_\infty$-weight, the Euclidean space with Euclidean
metric need not satisfy a $1$-Poincar\'e inequality. However, there is a natural metric induced by the strong
$A_\infty$-weight, and we show in Section~6 
that the Euclidean space equipped with this natural metric and weighted measure
satisfy a $1$-Poincar\'e inequality. Hence we are able to use the theory developed in the first part
to study rectifiability issues of the boundary of quasiminimal sets in this modified Euclidean space. We 
consider this application in Section~7 of this paper. It is known that every strong $A_\infty$-weight is
not comparable to the Jacobian of a Euclidean quasiconformal mapping; it is therefore not possible to
use (unweighted) Euclidean results about regularity of sets with quasiminimal surfaces to study rectifiability
issues of boundaries of such sets in the strong $A_\infty$-weighted setting. We were able to apply the
theory developed in the general metric setting in the first five sections of this paper to successfully address
rectifiability issues in this weighted Euclidean setting.

For related results about isoperimetric sets in the Carnot group setting we refer the
interested reader to~\cite{LR}, where they show that isoperimetric sets (which are necessarily a special class of sets of
quasiminimal boundary surfaces) are Ahlfors regular (which also now follows from Corollary~\ref{Minkowski}) and are porous.
Regularity for Euclidean quasiminimizers that are asymptotically minimizers was studied by Rigot~\cite{R}, where it was shown that if the asymptotic minimality condition is sufficiently controlled, then the quasiminimal surface is H\"older 
smooth in big pieces. We point out that our results about the sets with quasiminimal boundary surfaces
apply to every boundary point of the set (of course, a modification of such a set on a measure zero subset would
still maintain quasiminimality while destroying the regularity at some boundary point; to avoid this trivial modification
we ensure, without loss of generality, that each point $x$ of the boundary of the set $E$ satisfies
$\mu(B(x,r)\cap E)>0$ and $\mu(B(x,r)\setminus E)>0$ for all $r>0$), and hence our results are weaker than
H\"older regularity of the boundary, but are applicable to each boundary point. Hence it might well be that the
studies related to rectifiability and weak tangents of
locally minimal surfaces in the Carnot group setting would be more approachable using the regularity properties
studied in this paper. It is a result of Lu and Wheeden~\cite{LW}  that Carnot groups are doubling and satisfy a
$1$-Poincar\'e inequality, and hence the results of this paper apply in the setting of Carnot groups (and indeed in
more general Carnot-Carath\'eodory spaces, which satisfy local versions of these conditions). A nice survey about
Poincar\'e inequalities and isoperimetric inequalities in the setting of Carnot groups can also be found in
\cite{J} and \cite{Hei}. 

It was shown in~\cite{AKL} that a subset $E$, of a Carnot group, with locally finite perimeter
has vertical weak tangents for $\Vert D\chi_E\Vert$-almost every point. Combining this with
our results (in particular, the consequence that every boundary point of such a set is in the measure-theoretic
boundary), we see that $\mathcal{H}^{Q-1}$-a.e.~boundary point of a set of quasiminimal boundary surface 
in a Carnot group has a vertical weak tangent ($Q$ is the 
homogeneous dimension of the group). The method of~\cite{AKL} uses the group structure; it would be
interesting to know whether such results hold for other Carnot-Carath\'eodory spaces such as
the Grushin spaces. Note that existence of weak tangents is weaker than rectifiability.
For a different notion of rectifiability in the Carnot group setting see~\cite[Section~3]{Mag}.

\section{Preliminaries}
A Borel regular outer measure is doubling if there is a constant $C>0$ such that 
for every ball $B\subset X$ we have $0<\mu(B)<\infty$ with $\mu(2B)\le C\, \mu(B)$.  For such a measure $\mu$, there
is a lower mass bound exponent $Q>0$; that is, whenever $x\in X$, $0<r\le R$, and $y\in B(x,R)$,  we have
\[
    \frac{\mu(B(y,r))}{\mu(B(x,R))}\ge \frac{1}{C}\, \left(\frac{r}{R}\right)^Q.
\]  

Given a function $f$
and a non-negative Borel measurable function $g$ on $X$, we say that $g$ is an upper gradient of $f$ if whenever
$\gamma$ is a rectifiable curve in $X$ (that is, a curve with finite length), we have
\begin{equation}\label{ugradineq}
   |f(y)-f(x)|\le \int_\gamma g\, ds,
\end{equation}
where $x$ and $y$ denote end points of $\gamma$. Here the above inequality should be interpreted to mean that
$\int_\gamma g\, ds=\infty$ whenever at least one of $|f(x)|$ and $|f(y)|$ is infinite; see for example~\cite{HeiK}.
The collection of all upper gradients, together, play the role of the modulus of the weak derivative
of a Sobolev function in the metric setting. We consider the norm 
\[
  \Vert f\Vert_{N^{1,1}(X)}:=\Vert f\Vert_{L^1(X)}+\inf_g\Vert g\Vert_{L^1(X)}
\]
with the infimum taken over all upper gradients $g$ of $f$. The Newton-Sobolev space considered in this paper is the space
\[
N^{1,1}(X)=\{f\, :\, \|f\|_{N^{1,1}(X)}<\infty\}/{\sim},
\]
where the equivalence relation $\sim$ is given by $f\sim h$ if and only if 
\[
\Vert f-h\Vert_{N^{1,1}(X)}=0.
\]

We say that $X$ supports a weak $(1,1)$-Poincar\'e inequality if there are constants $C>0$ and $\lambda\ge 1$
such that whenever $f$ is a function on $X$ with upper gradient $g$ and $B$ is a ball in $X$, we have 
\[
   \int_B|f-f_B|\, d\mu\le C\text{rad}(B)\, \int_{\lambda B}g\, d\mu.
\]
A function $f$ on $X$ is said to be of bounded variation, and denoted $f\in BV(X)$, if $f\in L^1(X)$ and there is a 
sequence $\{f_n\}_n$ of functions from $N^{1,1}(X)$ such that $f_n\to f$ in $L^1(X)$ and 
$\limsup_n\Vert f_n\Vert_{N^{1,1}(X)}<\infty$. The $BV$ norm of such a function $f$ is given by
\[
   \Vert f\Vert_{BV(X)}:=\inf_{\{f_n\}_n}\liminf_{n\to\infty} \Vert f_n\Vert_{N^{1,1}(X)},
\]
where the infimum is taken over all such convergent sequences. The $BV$ energy norm of $f$ is given by
\[
  \Vert Df\Vert(X):=\inf_{\{f_n\}_n}\liminf_{n\to\infty}\bigg[ \Vert f_n\Vert_{N^{1,1}(X)}-\Vert f_n\Vert_{L^1(X)}\bigg].
\]
We say that a Borel set $E\subset X$ is of \emph{finite perimeter}   
if $\chi_E\in BV(X)$. The \emph{perimeter measure} of 
the set $E$ is $P(E,X):=\Vert D\chi_E\Vert(X)$. See~\cite{Mi2} and~\cite{A} for more on $BV$ functions
and sets of finite perimeter in the metric setting. 
We point out here that in the Euclidean case with Lebesgue measure the above notion coincides with
the classical definition of $BV$ functions; see for example~\cite{EG}. It was shown by Miranda in~\cite{Mi2} 
that if $\mathcal{U}$
is the collection of all open subsets of $X$, then the map
$\mathcal{U}\ni O\mapsto \Vert Df\Vert(O)$ extends to a Radon measure on $X$. The  \emph{coarea  formula} 
\[
   \Vert Df\Vert(A)=\int_{-\infty}^\infty P(\{x\in X\, :\, f(x)>t\}, A)\, dt
\]
was also proven in~\cite{Mi2}.

In this paper, we assume that $\mu$ is a doubling Borel measure with lower mass bound exponent $Q>1$ and that 
$X$ is complete and 
supports a $(1,1)$-Poincar\'e inequality.  
Note that we can increase the value of $Q$ as we like, and so assuming $Q>1$ is not a serious restriction, and is assumed
merely for book-keeping.
We point out here that if $X$ supports a weak $(1,1)$-Poincar\'e 
inequality, then whenever $f\in BV(X)$ and $B$ is a ball in $X$, we have
\[
   \int_B|f-f_B|\, d\mu\le C\, \text{rad}(B)\, \Vert Df\Vert(\lambda B).
\] 
When considering the function $f=\chi_E$ for set $E\subset X$, the above inequality implies the 
\emph{ relative isoperimetric inequality} 
\[
    \min\{\mu(B\cap E),\mu(B\setminus E)\}\le C\, \text{rad}(B)\, P(E,\lambda B).
\]
In this paper $C$ will denote constants whose precise values are not needed, and so the
value of $C$ might differ even within the same line. 

It is well known that the Poincar\'e inequality implies a Sobolev-Poincar\'e inequality
if the measure is doubling. 
Indeed, by~\cite{HaKo} we have 
\[
   \Big(\barint_B|u-u_B|^t\, d\mu\Big)^{1/t}\le C \text{rad}(B)\barint_{\lambda B}g_u\, d\mu.
\]
with $t=Q/(Q-1)$ for all $u\in N^{1,1}(X)$.
By the definition of the $BV$ class, lower semicontinuity of the $BV$ norm, and the Lebesgue dominated convergence
theorem, we obtain the Sobolev inequality
\[
   \Big(\barint_B|u-u_B|^t\, d\mu\Big)^{1/t}\le C {\rm rad}(B)\frac{\Vert Du\Vert(2\lambda B)}
                                                                       {\mu(2\lambda B)}
\]
for all $u\in BV(X)$.
 
\begin{lemma}\label{Lemma2.1}
 Let $u\in BV(X)$ and $A=\{x\in B:|u(x)|>0\}$. If $\mu(A)\le \gamma\mu(B)$ for some $0<\gamma<1$,
then 
\[
  \Big(\barint_B|u|^t\, d\mu\Big)^{1/t}
        \le \frac{C}{1-\gamma^{1-1/t}}{\rm rad}(B)\frac{\Vert Du\Vert(2\lambda B)}
                                                                       {\mu(2\lambda B)}
\]
with $t=Q/(Q-1)$.
\end{lemma}

\begin{proof}
By Minkowski's inequality and the above-mentioned Sobolev inequality,
\begin{equation}\label{eq:1}
\begin{split}
  \Big(\barint_B|u|^t\, d\mu \Big)^{1/t}
  &\le \Big(\barint_B|u-u_B|^t\, d\mu \Big)^{1/t}+|u_B|\\
                  &\le C \text{rad}(B)\frac{\Vert Du\Vert(2\lambda B)}{\mu(2\lambda B)}+|u_B|.
\end{split}
\end{equation}
By the assumption on $u$ and H\"older's inequality,
\begin{align*}
  |u_B|&\le \barint_B|u|\, d\mu  =\frac{1}{\mu(B)}\, \int_A|u|\, d\mu\\
                               & \le \frac{1}{\mu(B)}\Big(\int_A|u|^t\, d\mu\Big)^{1/t}
                                     \mu(A)^{1-1/t}\\
                               & = \frac{1}{\mu(B)}\Big(\int_B|u|^t\, d\mu\Big)^{1/t}
                                     \mu(A)^{1-1/t}\\
                               & = \Big(\frac{\mu(A)}{\mu(B)}\Big)^{1-1/t}
                                      \Big(\barint_B |u|^t\, d\mu\Big)^{1/t} \\
                               & \le \gamma^{1-1/t}\Big(\barint_B |u|^t\, d\mu\Big)^{1/t}.
\end{align*}
So by~\eqref{eq:1}, 
\[
   (1-\gamma^{1-1/t})\Big(\barint_B |u|^t\, d\mu\Big)^{1/t}
                             \le C \, \text{rad}(B)\frac{\Vert Du\Vert(2\lambda B)}{\mu(2\lambda B)},
\]
from which the lemma follows.
\end{proof}

\begin{corollary}\label{SobPoin}
If $u\in BV(X)$ such that $u=0$ in $X\setminus B$ and $X\setminus 2B$ is non-empty, then 
\[
   \Big(\barint_B|u|^t\, d\mu\Big)^{1/t}\le C {\rm rad}(B)\frac{\Vert Du\Vert(\overline{B})}{\mu(B)}.
\]
\end{corollary}

\begin{proof}
Since $X\setminus 2B$ is non-empty, and because by the Poincar\'e inequality $X$ is path-connected, it 
follows that there is a point $y\in 2B\setminus B$ such that $d(y,x)=3r/2$ where $B=B(x,r)$. Therefore by
the doubling property of the measure $\mu$, we have 
\[
\mu(2B\setminus B)\ge \mu(B)/C\ge \mu(2B)/C^2
\]
for some constant $C>1$.
Because $A=\{z\in 2B\, :\, |u(z)|>0\}$ is a subset of $B$, it follows that
\[
  \frac{\mu(A)}{\mu(2B)}\le \frac{\mu(B)}{\mu(2B)}=\frac{\mu(2B)-\mu(2B\setminus B)}{\mu(2B)}
                        \le 1-C^{-2}<1.
\]
We can take $\gamma=1-C^{-2}$ in Lemma~\ref{Lemma2.1} to obtain the desired inequality.
\end{proof}

\section{Quasiminimizing surfaces and quasiminimizers}

\begin{definition}
Let $E\subset X$ be a Borel set of finite perimeter and $\Om\subset X$ be an open set. We say that
$E$ is a $K$-quasiminimal set, or has a \emph{$K$-quasiminimal boundary surface}, in $\Om$  
if for all open $U\Subset\Om$ and for all Borel sets $F,G\Subset U$,
\[
   P(E,U)\le K\, P((E\cup F)\setminus G, U).
\]
We say that a function $u\in BV(\Om)$ is a $K$-quasiminimizer if for all $\pip\in BV(\Om)$ with
support in $U\Subset\Om$,
\[
   \Vert Du\Vert (U)\le K\, \Vert D(u+\pip)\Vert(U).
\]
\end{definition}

\begin{lemma}\label{star}
 If $E$ is a $K$-quasiminimal set in $\Om$, then $u=\chi_E\vert_{\Om}$ is a $K$-quasiminimizer in $\Om$.
\end{lemma}

\begin{proof}
Since $E$ is of finite perimeter, it follows that $u\in BV(\Om)$. Also, if $\pip\in BV(\Om)$
with support in  $U\Subset\Om$, then for $0<t<1$, 
when $x\in X\setminus U$ we have $(u+\pip)(x)>t$ if and only if $x\in E$, and consequently
\[
  P(\{u+\pip>t\},U)\ge K^{-1}\, P(E,U),
\]
and so by the coarea formula,  
\begin{align*}
  \Vert Du\Vert(U)&=P(E,U)=\int_0^1P(E,U)\, dt\\
                        &\le K\int_0^1 P(\{u+\pip>t\},U)\, dt\\
                        &\le K\int_\R P(\{u+\pip>t\},U)\, dt=\, K\, \Vert D(u+\pip)\Vert(U),
\end{align*}
which shows that $u$ is a $K$-quasiminimizer.
\end{proof}

\section{Density}

The main result of this section is Theorem~\ref{theorem:density}, where we prove a uniform measure density estimate 
for quasiminimal sets. To prove the main result, we need the following lemma. For a proof of this lemma, we refer 
to~\cite[Lemma~5.1]{Gia}.

\begin{lemma}\label{lemma:g}
Let $R > 0$ and $f : (0,R] \rightarrow [0,1)$ be a bounded function. Suppose
that there exist some $\alpha> 0$, $0\leq\theta< 1$, and $\gamma\geq0$ such that for all
$0 <\rho< r\leq R < \infty$ we have
\begin{equation*}
 f(\rho)\leq\gamma (r-\rho)^{-\alpha} +\theta f(r).
 \end{equation*}
Then there is a constant $c = c(\alpha,\theta)$ so that for all $0 < \rho< r\leq R$,
\begin{equation*}
 f(\rho)\leq c \gamma(r -\rho)^{-\alpha}.
 \end{equation*}
\end{lemma}

The next result implies that every boundary point of a set of quasiminimal surface
belongs to the measure theoretic boundary.

\begin{theorem}\label{theorem:density}
If $E$ is a  quasiminimal set in $\Om$, then by modifying $E$ on a set of measure
zero if necessary, there exists $\gamma_0>0$ such that for all $z\in\Om\cap\partial E$,
\[
  \frac{\mu(B(x,r)\cap E)}{\mu(B(x,r))}\ge\gamma_0
  \quad\text{and}\quad
  \frac{\mu(B(x,r)\setminus E)}{\mu(B(x,r))}\ge\gamma_0
\]
whenever $0<r<{\rm diam}(X)/3$ such that $B(x,2r)\subset\Om$.
The density constant $\gamma_0$ depends solely on the doubling constant, the constants 
associated with the Poincar\'e inequality, and the quasiminimality constant $K$.
\end{theorem}

\begin{proof}
We can modify $E$ on a set of measure zero so that $\mu(B(x,r)\cap E)>0$ for all $x\in E$ and $r>0$, and 
$\mu(B(x,r)\setminus E)>0$ for all $x\in X\setminus E$ and $r>0$. This is done by removing points $x\in E$ for which
there is a positive number $r_x$ such that $\mu(B(x,r_x)\cap E)=0$ (and in doing so, note that we remove the
ball $B(x,r_x)$ from $E$ as well since all points in this ball also satisfy this condition) and adding into $E$
points $y$ for which there is a positive number $r_y$ such that $\mu(B(y,r_y)\setminus E)=0$ (and in doing so,
note that we include the ball $B(y,r_y)$ back into $E$). By Lebesgue differentiation theorem, such a modification is
done only on a set of $\mu$-measure zero.
This implies that for all $x\in \partial E$ and $r>0$, we have 
\[
\mu(B(x,r)\cap E)>0
\quad\text{and}\quad
\mu(B(x,r)\setminus E)>0.
\] 
By the relative isoperimetric inequality, 
we conclude that
\[
P(E,B(x,r))>0.
\]

Let $u=\chi_E$, and for $z\in\Om$ let $R>0$ such that  $B(z,2R)\subset\Om$  with $0<R<\text{diam}(X)/3$.  
For $0<r<R$ let $\eta$ be a $C/(R-r)$--Lipschitz continuous function such that $\eta=1$ on $B(z,r)$ and
$\eta=0$ on $X\setminus B(z,R)$, with $0\le \eta\le 1$ on $X$. Set 
\[
v=u-\eta u=(1-\eta)u.
\] 
Then $v=u$ on $X\setminus B(z,R)$, and so by the quasiminimality property of $u$ and the product rule
\begin{align*}
  \Vert Du\Vert(B(z,r))&\le \Vert Du\Vert(B(z,R))\le K\Vert Dv\Vert(B(z,R))\\
                       &\le K\left(\Vert Du\Vert(B(z,R)\setminus B(z,r))
                                     +\frac{C}{R-r}\int_{B(z,R)}u\, d\mu\right).
\end{align*}
Observe that $\eta$ is a bounded Lipschitz function and so the product rule is valid.
By setting $\theta=K/(K+1)<1$, we see that
\[
  \Vert Du\Vert (B(z,r))\le \theta \Vert Du\Vert (B(z,R))+\frac{C}{R-r}\int_{B(z,R)}u\, d\mu.
\]
Hence by Lemma~\ref{lemma:g}, there is a constant $C>0$, which is independent of $z,R$ and $E$, such that
\[
  \Vert Du\Vert (B(z,r))\le \frac{C}{R-r}\int_{B(z,R)}u\, d\mu 
                           =\frac{C}{R-r}\, \mu(E\cap B(z,R)).
\]
For $r=\tfrac34R$, from the above we get
\begin{equation}\label{eq:2}
  \Vert Du\Vert (B(z,\tfrac34R))\le\frac{2C}{R}\, \mu(B(z,R)\cap E).
\end{equation}
Let $\nu$ be a $2C/R$--Lipschitz function such that $0\le\nu\le1$ on $X$, $\nu=1$ on
$B(z,\tfrac12R)$, and $\nu=0$ on $X\setminus B(z,\tfrac34R)$. Setting $\pip=\nu u$, 
the product rule implies that
\[
  \Vert D\pip\Vert (B(z,\tfrac34R))\le \Vert Du\Vert(B(z,\tfrac34R))
                                           +\frac{2C}{R}\mu(E\cap B(z,R)).
\]
So by~\eqref{eq:2}, we arrive at
\begin{equation}\label{eq:3}
  \Vert D\pip\Vert (B(z,\tfrac34R))\le \frac{2C}{R}\mu(E\cap B(z,R)).
\end{equation}

Notice that $\varphi^t=\varphi=\chi_E$ in $B(x,R/2)$ and therefore by 
Corollary~\ref{SobPoin}  and~\eqref{eq:3}, we obtain
\begin{equation}\label{eq:6}
\begin{split}
   \left( \frac{\mu (B(z,\tfrac R2)\cap E) }{ \mu(B(z,\tfrac R2)) }  \right)^{1-1/Q}
   =& \Big( \barint_{B(z,\tfrac R2)}\varphi^t\,d\mu \Big)^{1/t}\\
   \leq & C R
   \frac{\Vert D\pip\Vert(B(z,\tfrac34R))}{\mu(B(z,\tfrac34R))}\\
             \le & C\, \frac{\mu(B(z,R)\cap E)}{\mu(B(z,R))}. 
\end{split}
\end{equation}
Applying the above argument also to $X\setminus E$, we see that 
\begin{equation}\label{eq:7}
   \left(\frac{\mu(B(z,\tfrac R2)\setminus E)}{\mu(B(z,\tfrac R2))} \right)^{1-1/Q}
             \le C\, \frac{\mu(B(z,R)\setminus E)}{\mu(B(z,R))}. 
\end{equation}

Up to now we have been using an adaptation of a part of the De Giorgi machinery. To complete the
proof we adapt the proof of~\cite[Lemma~3.30]{DS1}.
Recall that by our assumption, if $x\in\Om\cap(E\cup\partial E)$ and $r>0$ then 
$\mu(B(x,r)\cap E)>0$.
For $x\in\Om\cap(E\cup\partial E)$ and $z\in B(x,\tfrac R4)$, by the doubling property of $\mu$, we have
\begin{equation}\label{eq:8}
  \frac{\mu(B(z,\tfrac R2)\cap E)}{\mu(B(z,\tfrac R2))}\le C_d \frac{\mu(B(x,R)\cap E)}{\mu(B(x,R))}.
\end{equation}
Let $\gamma_0=1/(C^QC_d)>0$, where $C$ is as in~\eqref{eq:6}. Suppose that
\[
   \frac{\mu(B(x,R)\cap E)}{\mu(B(x,R))}=\gamma<\gamma_0.
\]
For positive integers $j$ we set $B_j=B(z,R/2^j)$. Then by a repeated application of~\eqref{eq:6},
with $t=Q/(Q-1)>1$, we obtain
\begin{align*}
  \frac{\mu(B_j\cap E)}{\mu(B_j)}&\le \left(C\, \frac{\mu(B_{j-1}\cap E)}{\mu(B_{j-1})}\right)^{Q/(Q-1)}\\
                    &\le C^{Q/(Q-1)}\left(C\, \frac{\mu(B_{j-2}\cap E)}{\mu(B_{j-2})}\right)^{(Q/(Q-1))^2}\\
                    &\le C^{t+t^2+\cdots+t^{j-1}}
                               \left(\frac{\mu(B_1\cap E)}{\mu(B_1)}\right)^{t^{j-1}}\\
                    &\le C^{Q\, t^{j-1}}(C_d\gamma)^{t^{j-1}}
                    = \left(C^Q C_d\gamma\right)^{t^{j-1}},
\end{align*}
where we also used~\eqref{eq:8}. Since $C^Q C_d\gamma<1$, it follows that for all $z\in B(x,R/4)$,
\[
  \liminf_{r\to 0}\frac{\mu(B(z,r)\cap E)}{\mu(B(z,r))}=0,
\]
and the Lebesgue differentiation theorem now implies that $\mu(B(x,R/4)\cap E)=0$,  
resulting in a contradiction.
Consequently, we have
\[
   \frac{\mu(B(x,R)\cap E)}{\mu(B(x,R))}\ge \gamma_0.
\]
A similar argument for $X\setminus E$ also gives
\[
   \frac{\mu(B(x,R)\setminus E)}{\mu(B(x,R))}\ge \gamma_0.
\]
This completes the proof.
\end{proof}


\section{Porosity}

By a result of David and Semmes~\cite{DS1}, sets with quasiminimal surfaces in the complement of two
disjoint cubes in the Euclidean space are uniform domains whose complements are also uniform
(and indeed, are isoperimetric sets). Whether quasiminimal surfaces must enclose uniform
domains is still open in the general metric setting, but now that we know such sets have each boundary point
as a point of density for both the set and its complement, we  
next show that these sets are uniformly locally porous.   
For us, the porosity is a reasonable weakening of the uniform domain condition.

By Theorem \ref{theorem:density}, without loss of generality we may assume that 
every point $x\in\Omega\cap\partial E$ has the property that 
\[
  \frac{\mu(B(x,r)\cap E)}{\mu(B(x,r))}\ge \gamma_0
  \quad\text{and}\quad
     \frac{\mu(B(x,r)\setminus E)}{\mu(B(x,r))}\ge \gamma_0
\]
whenever $0<r<\text{diam}(X)/3$ such that $B(x,2r)\subset \Omega$,

\begin{lemma}\label{upperGrowth}
Let $E$ be a quasiminimal set in $\Omega$ and $x\in\Omega\cap\partial E$. 
Then there exist $r_0$ and $C>0$ such that
\[
   C^{-1}\, \frac{\mu(B(x,r))}{r}\le P(E,B(x,r))\le C\, \frac{\mu(B(x,r))}{r},
\] 
whenever $0<r<r_0$ such that $B(x,2r)\subset\Omega$.
The constant $C$ is independent of $x$ and $r$.  
\end{lemma}

\begin{proof}
The inequality on the left-hand side follows immediately from the density property of
both $E$ and $X\setminus E$ together with the relative isoperimetric inequality, so it suffices to prove the
inequality on the right-hand side.

By the results in~\cite[Lemma~6.2]{KKST1}, we have that for all $r>0$ there exists 
$r<\rho<2r$ (indeed, a positive $1$-dimensional measure amount of them) such that
\[
P(B(x,\rho))\approx\frac{\mu(B(x,\rho))}{\rho}
\]
and we can also choose such $\rho$ so that 
$P(E,S(x,\rho))=0$, where 
\[
S(x,\rho)=\{z\in X:d(z,x)=\rho\}
\] 
is the sphere centered at $x$ with radius $\rho$. Fix $\eps>0$.
Then $B(x,r)\subset \overline{B(x,\rho)}\subset B(x,\rho+\eps)$, and so by the quasiminimizer property of $E$ we have
\begin{align*}
  P&(E,B(x,r))\le P(E,B(x,\rho+\eps))\le K\, P(E\cup B(x,\rho),B(x,\rho+\eps))\\
             &\le K\, \left[ P(B(x,\rho),B(x,\rho+\eps)) + P(E,B(x,\rho+\eps)\setminus B(x,\rho-\eps))\right]\\
             &= K\, \left[ P(B(x,\rho)) + P(E,B(x,\rho+\eps)\setminus B(x,\rho-\eps))\right].
\end{align*}
Since $P(E,S(x,\rho))=0$, we have that 
\[
  \lim_{\eps\to 0}P(E,B(x,\rho+\eps)\setminus B(x,\rho-\eps))=0.
\] 
It follows from the choice of $\rho$ and the doubling property of $\mu$ that 
\[
  P(E,B(x,r))\le K\, P(B(x,\rho)) \approx K\frac{\mu(B(x,\rho))}{\rho}\approx C K\frac{\mu(B(x,r))}{r}.\qedhere
\]
\end{proof}

\begin{theorem}\label{porous1}
If $E$ is a quasiminimal set in $\Omega$, then $E$ and $X\setminus E$ are 
locally porous in $\Omega$; that is,
for every $x\in\Omega\cap\partial E$ there exists $r_x>0$ and $C\ge 1$ such that whenever
$0<r<r_x$, there are points $y\in B(x,r)$ and $z\in B(x,r)$  
such that 
\[
B(y,r/C)\subset E\cap\Omega 
\quad\text{and}\quad 
B(z,r/C)\subset X\setminus E.
\]
The constant $C$ is independent of $x,r$. Furthermore, $r_x$ depends on $x$ only so
far as to have $B(x,10r_x)\subset\Omega$.
\end{theorem}

\begin{proof}
Fix $x\in\Omega\cap\partial E$. For $r>0$ such that $B(x,4r)\subset\Omega$, let $0<\rho\le r$ such that
for all $y\in B(x,r)\cap E$ the ball $B(y,\rho)$ intersects $X\setminus E$. Note that $\rho=r$ would
satisfy this requirement. If there is some $\rho$ with $r/20<\rho<r/10$ such that the above condition fails,
then there is some $y\in B(x,r)\cap E$ such that $B(y,r/20)\subset E$, and the porosity requirement is
satisfied at the scale $r$. If not, we can choose $\rho<r/10$ so that for every $y\in B(x,r)\cap E$,
the set $B(y,\rho)\setminus E$ is non-empty. In this case, we can cover $B(x,r)\cap E$ by a family
of balls $\{B(y_i,10\rho)\}_i$, such that the collection $\{B(y_i,2\rho)\}_i$ is pairwise disjoint.
Then by the doubling property of $\mu$ together with
the density property of the previous section,
\begin{align*}
  \gamma_0\, \mu(B(x,r)) 
  &\le \mu(B(x,r)\cap E)
  \\&\le \sum_i\mu(B(y_i,10\rho))
 \le C\sum_i\mu(B(y_i,\rho)).
\end{align*}
Note that by the density results of the previous section,
\[
  \mu(B(y_i,2\rho)\cap E)\ge C\, \mu(B(y_i,2\rho))
\]
and 
\[
\mu(B(y_i,2\rho)\setminus E)\ge C\, \mu(B(y_i,2\rho)).
\]
Hence by the relative isoperimetric inequality,
\[
   P(E, B(y_i,2\rho))\ge \frac{1}{C}\frac{\mu(B(y_i,\rho))}{\rho}.
\]
By the pairwise disjointness property of the above family of balls and the
relative isoperimetric inequality combined with the density property of the 
previous section and the choice of $\rho$, we have
\begin{align*}
  P(E,B(x,2r))&\ge \sum_i P(E,B(y_i,2\rho))\\
              &\ge \sum_i\frac{1}{C}\frac{\mu(B(y_i,\rho))}{\rho}
\ge\frac{1}{C} \frac{\mu(B(x,r))}{\rho}.
\end{align*}
By Lemma~\ref{upperGrowth}, we now have
\[
\frac{1}{C}\frac{\mu(B(x,r))}{\rho}\leq C \frac{\mu(B(x,r))}{r},
\]
and consequently $\rho\geq r/C$.
This means that there is a point $y\in B(x,r)\cap E$ such that 
$B(y,r/2C)\subset E$, 
 thus proving the porosity of $E$.
Similar argument with $X\setminus E$, which also is a quasiminimal set since $E$ is a
quasiminimal set,  gives the porosity of $X\setminus E$ in $\Omega$.
\end{proof}

The following corollary is a consequence of the porosity property proved above. Note that in the Euclidean setting, if
a set satisfies the conclusion of the following corollary, then it is uniformly rectifiable; see for example the discussion
in~\cite{DS1}.  Indeed, David and Semmes use this fact together with the notion of tangent hyperplanes 
to prove that $E$ then has to be locally a John domain.

As a consequence of the following corollary together with the results from~\cite[Theorem~4.1]{LT}, the
Assouad dimension of $\partial E$ is at most $Q-1$, and by~\cite[Theorem~4.2]{LT}, the Assouad dimension of
$\partial E$ is $Q-1$ if the measure $\mu$ is
Ahlfors $Q$-regular, that is, 
\[
\mu(B)\approx \text{rad}(B)^Q. 
\]
In~\cite{LT} the supremum of all such possible $\alpha$
is called the Aikawa co-dimension of $\partial E$.
We also refer to \cite{LT} for the definition of the Minkowski content of codimension $\alpha$.
We denote $\delta_E(x)=\dist(x,X\setminus E)$.

\begin{corollary}\label{Minkowski}
If $E$ a quasiminimal set in domain $\Om$, then, then $\Om\cap \partial E$ has finite Minkowski content of codimension $\alpha$
for $0<\alpha<1$, and 
\[
\int_{B(x_0,r)\cap E}\ \frac{1}{\delta_E(y)^\alpha}\, d\mu(y) \le C\, \frac{\mu(B(x_0,r))}{r^\alpha}
\]
for all $x_0\in\partial E$ and $r>0$ such that $B(x_0,10\lambda r)\subset \Omega$.
Furthermore, if $\alpha\ge 1$ then 
\[
    \int_{B(x_0,r)\cap E}\ \frac{1}{\delta_E(y)^\alpha}\, d\mu(y)=\infty.
\]
\end{corollary}

\begin{proof}
By the Cavalieri principle, we see that 
\begin{align*}
   &\int_{B(x_0,r)\cap E}\frac{1}{\delta_E(y)^\alpha}\, d\mu(y)\\
        &=\int_0^\infty\mu\left(\{y\in B(x_0,r)\cap E\, :\, \delta_E(y)^{-\alpha}>t\}\right)\, dt\\
        &\approx \int_0^\infty \mu\left(\{y\in B(x_0,r)\cap E\, :\, \delta_E(y)<s\}\right)\, \frac{ds}{s^{1+\alpha}}\\  
        &\approx \int_0^r\mu\left(\{y\in B(x_0,r)\cap E\, :\, \delta_E(y)<s\}\right)\, \frac{ds}{s^{1+\alpha}}\\  
           &\qquad\qquad    +\int_{r}^\infty\frac{\mu(E\cap B(x_0,r))}{s^{1+\alpha}}\, ds\\
        &\approx \int_0^r\frac{\mu(E_s^+\cap B(x_0,r))}{s^{1+\alpha}}\, ds+\, C\, \frac{\mu(B(x_0,r))}{r^\alpha}.
\end{align*}
Here 
\[
E_s^+=\bigcup_{x\in \partial E}B(x,s)\cap E.
\]
To compute the measure of $E_s^+\cap B(x_0,r)$, we can cover
$E_s^+\cap B(x_0,r)$ by countably many balls $5\lambda B_j$ with radius $5\lambda s$, such that $\lambda B_j$ are pairwise disjoint. We also ensure that $5\lambda B_j$ has its center located in $B(x_0,r)\cap\partial E$. Now  
we have by the relative isoperimetric inequality and the porosity of $E$ and $\Omega\setminus E$ that  
\[
   \mu(B_j)\le C\, s\, P(E, \lambda B_j).  
\]
Thus by the doubling property of $\mu$ we conclude that
\begin{align*}
   \mu(E_s^+\cap B(x_0,r))
   &\le \sum_j\mu(5\lambda B_j)
   \le C\, \sum_j\mu(B_j)  
   \\&\le Cs\, \sum_j P(E,\lambda B_j)  
        \le Cs\, P(E,B(x_0,2\lambda r)).
\end{align*}
By Lemma~\ref{upperGrowth},  we know that 
\[
P(E,B(x_0,2\lambda r))\approx \frac{\mu(B(x_0,r))}r.
\]
Hence we can conclude that
\[
   \int_0^r\frac{\mu(E_s^+\cap B(x_0,r))}{s^{1+\alpha}}\, ds\le C\, \frac{\mu(B(x_0,r))}{r}\int_0^r\frac{ds}{s^\alpha},
\]
and so
\[
  \int_{B(x_0,r)\cap E}\frac{1}{\delta_E(y)^\alpha}\, d\mu(y)\le C\, \frac{\mu(B(x_0,r))}{r^\alpha}.
\]
To see the second part of the claim, we can use the fact that the balls $\lambda B_j$ are pairwise disjoint, together with
Lemma~\ref{upperGrowth} from which we get that $\mu(B_j)\approx s\, P(E,\lambda B_j)$, to obtain 
\[
   \mu(E_s^+\cap B(x_0,r))\ge \frac{1}{C}\sum_{\{j: 5\lambda B_j\subset B(x_0,r)\}} \mu(B_j)
          \ge \frac{s}{C}P(E,B(x_0,\tfrac{r}{5\lambda})),
\]
from which we see that when $\alpha\ge 1$,
\[
  \int_{B(x_0,r)\cap E}\frac{1}{\delta_E(y)^\alpha}\, d\mu(y)\ge \frac{1}{C}P(E,B(x_0,\tfrac{1}{5\lambda}r))\, \int_0^r\frac{ds}{s^\alpha}=\infty.
\]
\end{proof}

\begin{remark}
The above proof also indicates that 
\[
     C^{-1}\, s\, P(E,B(x_0,\tfrac{1}{5\lambda}r))\le \mu(E_s^+\cap B(x_0,r))\le C\, s\, P(E,B(x_0,2\lambda r))
\]
whenever $x_0\in\partial E$ and $B(x_0,10\lambda r)\subset\Omega$.
\end{remark}

If $E$ is a quasiminimal set, then by the density property of the previous section,  
we have
$\overline{E}=\overline{\text{int}(E)}$, and so we can replace $E$ with its interior $\text{int}(E)$.
So we may assume that $E$ is open. From the density property again, we can see that connected
components of $E$ are also quasiminimal sets in $\Omega$. So throughout the rest of the paper we will assume that
in addition $E$ is a connected open set, that is, a domain, of (uniformly locally) quasiminimal surface.

We say that a domain $E$ is a $BV_l$-extension domain
(in the sense of ~\cite{BuMa} and~\cite{BaMo}) if there are constants $C\ge 1$ and $\delta>0$ such that 
whenever $u\in BV(E)$ such that the diameter of the support of $u$ is smaller than $\delta$, then there is a 
function $Tu\in BV(X)$ such that 
\[
\Vert D\, Tu\Vert(X)\le C\, \Vert Du\Vert(E)
\] 
and $Tu=u$ on $E$.

\begin{lemma}
If $E$ is a domain of (uniformly) locally quasiminimal surface, then it is a $BV_l$-extension domain. 
\end{lemma}

\begin{proof}
Let the uniform quasiminimality constant be denoted $K\ge 1$. Let $F\subset E$ be a set of finite perimeter in  $E$ 
with $\text{diam}(F)<\delta$. By the criterion of $BV_l$-extension domains found in~\cite{BuMa} and~\cite{BaMo},
to verify that $E$ is a $BV_l$-extension domain it suffices to prove that we can extend every
such $F$ to a set $\widehat{F}$ such that 
\[
P(\widehat{F},X)\le (1+K)P(F,E).
\] 
Indeed, via approximation of $BV$ functions by Lipschitz functions
and by the coarea formula, it suffices to prove the above claim for relatively closed (in $E$) and bounded subsets $F$ of $E$. We will
show that we can choose $\widehat{F}=F$.

To see this, note that 
\[
P(F,X)=P(F,E)+P(F,\partial E),
\] 
and so we need to show that 
\[
P(F,\partial E)\le K\, P(F,E).
\]

Fix $\varepsilon>0$. Let $\Omega$ be a relatively compact open subset of $X$ such that $F\Subset\Omega$. 
We can take $\Omega$ sufficiently small so that $P(E,\Omega\setminus \overline F)<\varepsilon$.

Then by the local quasiminimality property of $E$, we see that
\begin{equation}\label{eqn:Eqm}
 P(E,\Omega)\leq K P(E\setminus F,\Omega).  
\end{equation}
Since
\[
P(E,\Omega)=P(E,\Omega\setminus\overline F)+P(E,\overline F)=P(E,\Omega\setminus\overline F)+P(F,\partial E)
\]
and
\[
 P(E\setminus F,\Omega)=P(E,\Omega\setminus \overline F)+P(E\setminus F,\overline F)=P(E,\Omega\setminus \overline F)+P(F,E),
\]
the estimate \eqref{eqn:Eqm} implies that
\[
 P(F,\partial E)\leq K P(F,E)+(K-1)\varepsilon.
\]
By letting $\varepsilon\rightarrow 0$, we obtain
\[
 P(F,\partial E)\leq K P(F,E).\qedhere
\]
\end{proof}

Recall that a domain $E$ is a local John domain if there exist constants $C, \delta>0$ such that whenever $x_0\in\partial E$
and $0<r<\delta$, for all points $x\in B(x_0,r)\cap E$ there is a point $y\in E\cap B(x_0,Cr)$ with $\delta_E(y)\ge r/C$ and
a curve $\gamma\subset E$, called a John curve, connecting $x$ to $y$ satisfying 
\[
   \ell(\gamma_{x,z})\le C\, \delta_E(z)
\]  
for all $z\in\gamma$; here $\gamma_{x,z}$ denotes a subcurve of $\gamma$ with end points $x$ and $z$.
We conclude this section with the following open question: if $E$ is a domain of locally quasiminimal surface, then
is it true that $E$ is a local John domain? In the Euclidean setting this question was answered in the affirmative by
David and Semmes~\cite{DS1}. The crucial part of the proof of~\cite{DS1} is to show that the boundary of a set of quasiminimal
surface lies locally close to a hyperplane; in the setting of metric measure spaces one does not have such a structure, and
the challenge is to construct an alternative approach.

\section{Support of Poincar\'e inequality in $(\R^n,d,\mu)$.}

Let $\omega$ be a strong $A_\infty$-weight on $\R^n$. 
Then we have that with the measure $\mu$ on $\R^n$ defined by   
the density condition
\[
d\mu(x)=\omega(x)\, d\mathcal{L}^n(x),
\]
there is a metric $d$ on $\R^n$ and a constant $C\ge1$ such that whenever $x,y\in\R^n$ and $B_{x,y}$ is the smallest
Euclidean ball in $\R^n$ containing $x$ and $y$ (that is, $B_{x,y}=B(\tfrac{x+y}{2},\tfrac{|x-y|}{2})$),  
\[
    \frac{1}{C}\mu(B_{x,y})^{1/n}\le d(x,y)\le C\mu(B_{x,y})^{1/n}.
\]
Since strong $A_\infty$-weights are $A_\infty$-weights,   
$\omega$ is a Muckenhoupt $A_p$-weight for some $p$. 
It follows that $\mu$ is a doubling measure with respect to the Euclidean metric.   
Hence, we have a constant $C\ge 1$ such that whenever $x,y\in\R^n$,
\begin{equation}\label{eq:1A}
   \frac{1}{C} \mu(B(x,|x-y|))\le d(x,y)^n\le C \mu(B(x,|x-y|)).
\end{equation}
For the definition and properties of strong $A_\infty$-weights, we refer to \cite{DS2}.

The metric space we consider here is $(\R^n,d,\mu)$. Balls in this metric are denoted with the superscript $d$ in order
to distinguish them from the Euclidean balls. So 
\[
B^d(x,r)=\{y\in\R^n\, :\, d(x,y)<r\},
\] 
while
\[
B(x,r)=\{y\in\R^n\, :\, |x-y|<r\}.
\] 
We note that the topology generated by the metric $d$ is the same one as the Euclidean topology.
 
 In this section we show that when the measure $\mu$ on $\R^n$ is given by a strong $A_\infty$-weight, then 
 the space $(\R^n,d,\mu)$ is an Ahlfors regular space supporting 
 a $(1,1)$-Poincar\'e inequality. Note that not all strong $A_\infty$-weights are $A_1$-weights, and so in general
 the space $(\R^n,|\cdot|, \mu)$ need not support a $(1,1)$-Poincar\'e inequality.
 The next result states that $(\R^n,d,\mu)$ is Ahlfors $n$-regular.
 
 \begin{lemma}\label{Ahlfors}
  There is a constant $C\ge1$ such that whenever $x\in\R^n$ and $r>0$, we have
\[
   \frac{1}{C} r^n\le \mu(B^d(x,r))\le Cr^n.
\]
\end{lemma}

\begin{proof}
Let $y\in\partial B^d(x,r)$ such that 
\[
|y-x|=\sup\{|x-z| : z\in\partial B^d(x,r)\}.
\]
Note that as $\overline{B}^d(x,r)$ is compact, such $y$ exists. Then $B^d(x,r)\subset B(x,|x-y|)$, and
so by~\eqref{eq:1A}, we have
\[
   \mu(B^d(x,r))\le\mu(B(x,|x-y|))\le C d(x,y)^n = C r^n.
\]

Next, let $z\in\partial B^d(x,r)$ such that 
\[
|x-z|=\inf\{|x-z|\, :\, z\in\partial B^d(x,r)\}.
\] 
Then
$B(x,|x-z|)\subset B^d(x,r)$, and so again by~\eqref{eq:1A} we have
\[
  \mu(B^d(x,r))\ge\mu(B(x,|x-z|))\ge \frac{1}{C}d(x,z)^n=\frac{1}{C}r^n,
\]
 completing the proof.
 \end{proof}
 
  \begin{lemma}\label{App0}
 There is a Borel set $F\subset\R^n$ with $|F|=0$ such that whenever
 $\gamma$ is a curve in $\R^n$ which is rectifiable with respect to the Euclidean metric,  it is
 rectifiable with respect to the metric $d$ if $\int_\gamma \infty\chi_F+\omega^{1/n}\, ds$ is finite.
 In this case we also have that the length of $\gamma$ with respect to the metric $d$, denoted $\ell_d(\gamma)$,
 satisfies
 \[
    \ell_d(\gamma)\approx \int_\gamma\omega^{1/n}\, ds.
 \]
 \end{lemma}
 
 \begin{proof}
  Fix $x\in\R^n$. Then for $y\in\R^n$, by~\eqref{eq:1A} we have
  \begin{align*}
    \frac{d(x,y)}{|x-y|}&\approx \frac{\mu(B(x,|x-y|))^{1/n}}{|x-y|}\\
                                  &\approx \frac{1}{|B(x,|x-y|)|^{1/n}}\, \left(\int_{B(x,|x-y|)}\omega(z)\, dz\right)^{1/n}\\
                                  &\approx\left(\frac{1}{|B(x,|x-y|)|}\, \int_{B(x,|x-y|)}\omega(z)\, dz\right)^{1/n}.
  \end{align*}
  Denote
  \[
  \underline{\rho}(x)= \liminf_{y\to x}\frac{d(x,y)}{|x-y|}  
  \quad\text{and}\quad
  \overline{\rho}(x)=\limsup_{y\to x}\frac{d(x,y)}{|x-y|}.   
  \]
  Since $\omega\in L^1_{\loc}(\R^n)$ (the integrals being taken with respect to the Lebesgue measure), we see
  by Lebesgue differentiation theorem that for almost every $x\in\R^n$,
  \begin{align*}
     \omega(x)^{1/n}&=\lim_{y\to x}\left(\frac{1}{|B(x,|x-y|)|}\, \int_{B(x,|x-y|)}\omega(z)\, dz\right)^{1/n}\\
          &\le C \underline{\rho}(x)
          \le C\overline{\rho}(x)\\
          &\le C^2\lim_{y\to x} \left(\frac{1}{|B(x,|x-y|)|}\, \int_{B(x,|x-y|)}\omega(z)\, dz\right)^{1/n}
          =C^2\omega(x)^{1/n}.
  \end{align*}
  Let $F$ be the set of all non-Lebesgue points of $\omega$; then $\mu(F)=|F|=0$.
  Let $\gamma$ be an Euclidean rectifiable curve with $\int_\gamma\infty\chi_F+\omega^{1/n}\, ds<\infty$.
  Then $\mathcal{H}^1(\gamma^{-1}(\gamma\cap F))=0$, and in addition we have
  \[
     \int_\gamma \underline{\rho}\, ds\le \ell_d(\gamma)\le \int_\gamma\overline{\rho}\, ds,
  \]
  where $\ell_d(\gamma)$ is the length of $\gamma$ in the metric $d$. It follows that
  \[
     \int_\gamma \omega^{1/n}\, ds\le C\ell_d(\gamma)\le C^2\, \int_\gamma\omega^{1/n}\, ds.\qedhere  
  \]
 \end{proof}
 
 \begin{lemma}\label{App1}
 If $u$ is Lipschitz continuous with respect to the metric $d$, then $u\in W^{1,n}_{\loc}(\R^n)$.
 \end{lemma}
 
 \begin{proof}
  Let $h\in\R$, and let $e_j$, $j=1,\dots,n$ denote the standard orthonormal basis for $\R^n$. Since 
  $u$ is Lipschitz with respect to the metric $d$, we see by~\eqref{eq:1A} that
  \[
     |u(x+he_j)-u(x)|\le Cd(x,x+he_j)\le C\mu(B(x,|h|))^{1/n}.
  \]
  Thus, for $a\in\R^n$, $R>0$ and $|h|\le R$, by Fubini's theorem we see that
  \begin{align*}
    \int_{B(a,R)}|u(x+he_j)&-u(x)|^n\, dx\le C\, \int_{B(a,R)}\mu(B(x,|h|))\, dx\\  
              &\le C\int_{B(a,R)}\, \int_{B(a,2R)}\chi_{B(x,|h|)}(y)\, d\mu(y)\, dx\\
              &= C\int_{\R^n}\int_{R^n}\chi_{B(a,R)}(x)\, \chi_{B(x,|h|)}(y)\, d\mu(y)\, dx\\  
              &= C\int_{R^n}\int_{R^n}\chi_{B(a,R)}(x)\, \chi_{B(y,|h|)}(x)\, dx\, d\mu(y)\\   
              &\le C |h|^n\, \int_{B(a,2R)}\, d\mu(y)
              = C|h|^n\, \mu(B(a,2R)).
  \end{align*}
  So whenever $\Omega\Subset\R^n$ is an open set and $\Omega^\prime\Subset\Omega$, we can cover
  $\Omega^\prime$ by a countable collection $\{B_i\}$
  of balls of radius $R=\text{dist}(\Omega^\prime,\R^n\setminus\Omega)/2$ centered at points in $\Omega^\prime$
  and with bounded overlap of the balls $\{2B_i\}$, to obtain that
  \begin{align*}
    \int_{\Omega^\prime}\frac{|u(x+he_j)-u(x)|^n}{|h|^n}\, dx
          &\le \sum_i\int_{B_i}\frac{|u(x+he_j)-u(x)|^n}{|h|^n}\, dx\\
                   &\le C\sum_i \mu(2B_i)\ \le C \mu(\Omega)<\infty.
  \end{align*}
  So by~\cite[Lemma~7.24]{GT}, we see that $u\in W^{1,n}_{\loc}(\R^n)$.  
 \end{proof}
 
 By Lemma~\ref{App1} we know that if $u$ is Lipschitz continuous with respect to the metric $d$, then
 it has an Euclidean weak derivative $\nabla u$, and that $|\nabla u|$ is a minimal $1$-weak upper gradient (in the
 Euclidean metric) of $u$.
 
 In order to prove that $(\R^n,d,\mu)$ supports a $(1,1)$-Poincar\'e inequality, it suffices to prove the inequality
 for Lipschitz functions with respect to $d$ and their \emph{continuous} upper gradients in the metric $d$; see
 for example~\cite{Ke}. 
 Let $u$ be a Lipschitz function with continuous upper gradient $g$ in $(\R^n, d,\mu)$.
 
 \begin{lemma}\label{App2}
   We have a constant $C>0$, that is independent of $g$ and $u$, such that
   \[
       |\nabla u(x)|\le C\omega(x)^{1/n}g(x)
   \]
   for almost every $x\in \R^n$.
 \end{lemma}
 
 \begin{proof}
  As in the proof of Lemma~\ref{App1}, we let $e_j$, $j=1,\dots,n$, denote the canonical orthonormal basis 
  of $\R^n$. For each $j=1,\dots,n$, we consider the collection $\Gamma_j$ of all line segments parallel to the direction of
  $e_j$. From~\cite[Section~7.2, page~21]{V}, we know that whenever $\Gamma\subset\Gamma_j$ satisfies 
  $\text{Mod}_1(\Gamma)=0$, we have 
  \[
  \Big|\bigcup_{\gamma\in\Gamma}\cup_{z\in\gamma}z\Big|=0.
  \]
  
  Let $\Gamma_a$ denote the collection of all compact line segments in $\Gamma_j$ for which 
  $\int_\gamma \infty\chi_F+\omega^{1/n}\, ds=\infty$ (with $F$ as in Lemma~\ref{App0}). Since by H\"older's 
  inequality we know that $\omega^{1/n}\in L^1_{\loc}(\R^n)$, it follows that $\text{Mod}_1(\Gamma_a)=0$;
  see for example~\cite{KoMc}.   
  
  Since $u$ is Lipschitz continuous with respect to $d$ and the topology
  induced by $d$ and the Euclidean topology coincide, we see that $u$ is continuous in the Euclidean space $\R^n$.
  By Lemma~\ref{App1} we also know that $u\in W^{1,1}_{\loc}(\R^n)$. 
  It follows from the discussion in~\cite{V} that $u$ is absolutely continuous on
  $\text{Mod}_1$-almost every compact Euclidean rectifiable curve in $\R^n$. Let $\Gamma_b$ denote the collection of
  all line segments $\gamma$ in $\Gamma_j$ along which $(u,|\nabla u|)$ does not support the upper gradient 
  inequality~\eqref{ugradineq}; that is, there is some sub-segment $\beta$ of $\gamma$ for which the 
  inequality~\eqref{ugradineq}
  fails. Since $|\nabla u|$ is a $1$-weak upper gradient of $u$, it follows that $\text{Mod}_1(\Gamma_b)=0$. Furthermore, let
  $\Gamma_c$ denote the collection of all segments $\gamma\in\Gamma_j$ for which $\int_\gamma|\nabla u|\, ds$ is
  infinite.
  
  Because $g$ is an upper gradient of $u$ in the metric $d$, by Lemma~\ref{App0} we know that whenever
  $\gamma\in\Gamma_j\setminus\Gamma_a$, for all sub-segments $\beta$ of $\gamma$ we have
  \[
     |u(x_\beta)-u(y_\beta)|\le C \int_\beta\omega^{1/n} g\, ds.
  \]
  Here $x_\beta$ and $y_\beta$ denote the two end points of $\beta$.
  It follows that if $\gamma\not\in\Gamma_b\cup\Gamma_c$ as well, then 
  for $\mathcal{H}^1$-almost every point $x\in\beta$, 
  \[
      |\partial_j u(x)|\le C\omega^{1/n}(x) g(x).
  \]
  Note that $\text{Mod}_p(\Gamma_a\cup\Gamma_b\cup\Gamma_c)=0$.
  Hence by the use of~\cite{V} again, we see that for almost every $x\in\R^n$ we have
  \[
      |\partial_j u(x)|\le C\omega^{1/n}(x) g(x).
  \]
  Now the conclusion follows by summing up over $j=1,\dots, n$.
 \end{proof}
 
 We next compare Euclidean balls with balls in the metric $d$.
 
 \begin{lemma}\label{App3}
 There is a constant $C>0$ such that whenever $x\in\R^n$ and $r>0$, there exist positive numbers
 $\lambda_x^r$ and $\tau_x^r$ such that
 \begin{equation}\label{eq:nested1}
   B(x,\lambda_x^r\, r)\subset B^d(x,r)\subset B(x,C\lambda_x^r\, r)
 \end{equation}
 and
 \begin{equation}\label{eq:nested2}
   B^d(x,\tau_x^r\, r)\subset B(x,r)\subset B^d(x,C\tau_x^r\, r).
 \end{equation}
 \end{lemma}
 
 \begin{proof}
 Since $\mu$ is a doubling measure on the Euclidean space $\R^n$ and $\R^n$ is uniformly perfect, 
 there is a constant $Q_1>0$ such that whenever $0<r<R$,
 \begin{equation}\label{UMB1}
    \frac{\mu(B(x,r))}{\mu(B(x,R))}\le C\, \left(\frac{r}{R}\right)^{Q_1}.
 \end{equation}
 Let $z_x^r, y_x^r\in\partial B^d(x,r)$ such that for $y\in\partial B^d(x,r)$ we have $|x-z_x^r|\le |x-y|\le |x-y_x^r|$.
 Then set 
 \[
   \lambda_x^r=\frac{|x-z_x^r|}{r}.
 \]
 We have 
 \[
 B(x,|x-z_x^r|)=B(x,\lambda_x^r\, r)\subset B^d(x,r)\subset B(x,|x-y_x^r|).
 \]
 Because of the upper mass bound~\eqref{UMB1}, by the twice-repeated use of~\eqref{eq:1A},
 \[
     \frac{1}{C}\le \frac{\mu(B(x,|x-z_x^r|))}{\mu(B(x,|x-y_x^r|))}
          \le C\,\left(\frac{|x-z_x^r|}{|x-y_x^r|}\right)^{Q_1},
 \]
 and so it follows that $|x-y_x^r|\le C|x-z_x^r|$, whence we obtain that 
 $B(x,|x-y_x^r|)\subset B(x,C|x-z_x^r|)$, and this proves~\eqref{eq:nested1}.
 
 To prove~\eqref{eq:nested2},  we  consider $w_x^r\in\partial B(x,r)$ such that
 $d(x,w_x^r)\le d(x,y)$ whenever $y\in\partial B(x,r)$, and set 
 \[
      \tau_x^r=\frac{d(x,w_x^r)}{r}.
 \]
 As in the previous argument, we consider also $a_x^r\in\partial B(x,r)$ such that 
 $d(x,a_x^r)\ge d(x,y)$ for all $y\in\partial B(x,r)$, and obtain by the use of Lemma~\ref{Ahlfors} that
 \[
     \frac{\mu(B^d(x,d(x,w_x^r)))}{\mu(B^d(x,d(x,a_x^r)))}\approx \left(\frac{d(x,w_x^r)}{d(x,a_x^r)}\right)^n,
 \]
 and from~\eqref{eq:1A} we also see that 
 \[
    \mu(B^d(x,d(x,w_x^r)))\approx d(x,w_x^r)^n\approx\mu(B(x,|x-w_x^r|))=\mu(B(x,r)).
 \]
 A similar argument as above also shows that
 \[
   \mu(B^d(x,d(x,a_x^r)))\approx\mu(B(x,r)).
 \]
 It follows that
 \[
    \left(\frac{d(x,w_x^r)}{d(x,a_x^r)}\right)^n\ge \frac{1}{C},
 \]
 that is, $d(x,a_x^r)\le C\, d(x,w_x^r)$. From this~\eqref{eq:nested2} follows.
 \end{proof}
 
 Now we are ready to prove the main result of this section.  
 
 \begin{proposition}\label{prop-PI-weights}
  The metric measure space $(\R^n, d, \mu)$ is an Ahlfors $n$-regular space supporting a $(1,1)$-Poincar\'e 
  inequality.
 \end{proposition}
 
 \begin{proof}
  As pointed out by~\cite{Ke}, it suffices to prove the inequality for functions $u$ that are Lipschitz continuous
  on $(\R^n,d)$ with continuous upper gradient $g$. By~\cite[Inequality~(1.10)]{DS2}, we know that when $x\in\R^n$
  and $r>0$,  
  \begin{align*}
    \int_{B(x,C\lambda_x^r r)}&\barint_{B(x,C\lambda_x^rr)}|u(x)-u(y)|\, d\mu(x)\, d\mu(y)\\
       &\le C\mu(B(x,C\lambda_x^r))^{1/n}\int_{B(x,2C\lambda_x^r\, r)}\omega(x)^{-1/n}\, |\nabla u(x)|\, d\mu(x).
  \end{align*}
  By the doubling property of $\mu$ on the Euclidean $\R^n$, Lemma~\ref{App2}, and Lemma~\ref{Ahlfors}, we see that
   \begin{align*}
    \int_{B(x,C\lambda_x^r\, r)}&\barint_{B(x,C\lambda_x^r r)}|u(x)-u(y)|\, d\mu(x)\, d\mu(y)\\
       &\le C\mu(B(x,C\lambda_x^r r))^{1/n}\, \int_{B(x,2C\lambda_x^r r)}g(x) \, d\mu(x)\\
       &\le C\mu(B(x,\lambda_x^rr))^{1/n}\, \int_{B(x,2C\lambda_x^r r)}g(x) \, d\mu(x)\\
       &\le C\mu(B^d(x,r))^{1/n}\int_{B(x,2C\lambda_x^r r)}g(x) \, d\mu(x)\\
       &\le Cr\int_{B(x,2C\lambda_x^r r)}g(x) \, d\mu(x).
  \end{align*}
  Hence,
  \begin{align*}
    \inf_{c\in\R}\int_{B^d(x,r)}|u-c|\, d\mu 
           &\le \int_{B^d(x,r)}|u-u_{B(x,C\lambda_x^r r)}|\, d\mu\\
           &\le \int_{B(x,C\lambda_x^r r)}|u-u_{B(x,C\lambda_x^r r)}|\, d\mu\\
           &\le C \int_{B(x,C\lambda_x^r r)}|u-u_{B(x,C\lambda_x^r r)}|\, d\mu\\
           &\le C  \int_{B(x,C\lambda_x^r\, r)}\barint_{B(x,C\lambda_x^r r)}|u(x)-u(y)|\, d\mu(x)\, d\mu(y)\\  
           &\le Cr\int_{B(x,2C\lambda_x^r r)}g(x) \, d\mu(x).
  \end{align*}
  By Lemma~\ref{App3} we have 
  \[
  B(x,2C\lambda_x^r r)\subset B^d(x,2C^2\tau_x^{2C\lambda_x^r r}\, \lambda_x^rr).
  \]
  Note that 
  $B^d(x,\tau_x^{2C\lambda_x^r r}\, 2C\lambda_x^r\, r)$ is the largest metric ball centered at $x$ that fits inside the
  Euclidean ball $B(x,2C\lambda_x^r r)$. Let $\rho>0$ such that $B^d(x,\rho)$ is the largest metric ball centered at $x$ and
  contained in the Euclidean ball $B(x,\lambda_x^rr)$, and let $y_1\in\partial B^d(x,\rho)\cap\partial B(x,\lambda_x^rr)$,
  and correspondingly let 
  $y_2\in\partial B^d(x,\tau_x^{2C\lambda_x^rr}\, 2C\lambda_x^r r)\cap\partial B(x,2C\lambda_x^r r)$. Then
  by~(\ref{eq:1A}),
  \[
    \rho=d(x,y_1)\approx\mu(B(x,\lambda_x^r r))^{1/n}
    \]
    and again by~(\ref{eq:1A}),  
    \[
    \tau_x^{2C\lambda_x^rr}\, 2C\lambda_x^r r=d(x,y_2)\approx\mu(B(x,2C\lambda_x^r r))^{1/n}.  
  \]
  By the doubling property of $\mu$ in the Euclidean space $\R^n$, we see that 
  \[
     \mu(B(x,\lambda_x^r r))\approx \mu(B(x, 2C\lambda_x^rr)).  
  \]
  It follows that $\rho\approx \tau_x^{2C\lambda_x^rr}\, 2C\lambda_x^rr$. On the other hand, since $B^d(x,\rho)$ is
  the largest metric ball centered at $x$ and fitting inside the Euclidean ball $B(x,\lambda_x^r r)$, and by
  the construction of $\lambda_x^r$ from Lemma~\ref{App3} we know that $B(x,\lambda_x^r r)\subset B^d(x,r)$, we
  can conclude that $\rho\le r$. Hence 
  \[
     \tau_x^{2C\lambda_x^r r}2C\lambda_x^r r\le C \rho\le C_2 r.
  \]
  Thus we have
  \[
     B(x,2C\lambda_x^r r)\subset B^d(x,2C^2\tau_x^{2C\lambda_x^r r}\, \lambda_x^r r)\subset B^d(x,C_2 r),
  \]
  from which we conclude that
  \[
  \begin{split}
         \inf_{c\in\R}\int_{B^d(x,r)}|u-c|\, d\mu
         &\le Cr \int_{B(x,2C\lambda_x^r r)}g(x) \, d\mu(x) \\
           & \le Cr\int_{B^d(x,C_2 r)}g\, d\mu,
 \end{split}
  \]
  which is equivalent to the $(1,1)$-Poincar\'e inequality on $(\R^n,d,\mu)$ because the constants $C$, $C_2$ are 
  independent of $x,r,u,g$.
 \end{proof}

\section{Rectifiability of quasiminimal surfaces in Euclidean spaces with strong $A_\infty$ weights}

In this section we apply the results from the earlier part of this paper (from the first three sections) to the setting
of weighted Euclidean spaces where the weight $\omega$ is a strong $A_\infty$-weight.

If $\omega$ is an $A_1$-weight, then the space $(\R^n, |\cdot|,\mu)$ is a doubling metric measure space supporting a
$(1,1)$-Poincar\'e inequality. However, not all strong $A_\infty$-weights are $A_1$-weights, but 
as shown in the previous section, 
the metric measure space $(\R^n, d, \mu)$ is also an Ahlfors $n$-regular space supporting a $(1,1)$-Poincar\'e inequality. In 
this section we will prove that a set $E\subset\R^n$ that has a locally quasiminimal boundary surface in 
$(\R^n,d,\mu)$ will have a rectifiable boundary. Here of course, the notion of rectifiability is in terms of the  
Euclidean metric. As in~\cite[page~204, Definition~15.3]{Mat}, we say that a set $A\subset\R^n$ is 
$m$-rectifiable if there is a countable collection $\{f_i\}$ of Euclidean Lipschitz maps $f_i:\R^m\to\R^n$ such that 
$\mathcal{H}^m_{Euc}(A\setminus \bigcup_if_i(\R^m))=0$. In this section, we consider the issue of whether 
the boundary $\partial E$ of the set with quasiminimal boundary surface in $(\R^n,d,\mu)$ is $(n-1)$-rectifiable 
in the above sense (with respect to the Euclidean metric). We also recall that a set $K\subset\R^n$ is 
purely $m$-unrectifiable if whenever $A\subset\R^n$ is $m$-rectifiable, we have 
$\mathcal{H}^m_{Euc}(K\cap A)=0$. 
 
 Let $E\subset\Omega\subset \R^n$ be a set of finite perimeter with locally quasiminimal 
 boundary surface with respect to the metric $d$ and
 measure $\mu$. Then the results obtained in the previous sections of this note apply to $E$. 
 So we may assume that $E=\text{int}(\overline{E})$.
 
 \begin{lemma}\label{Tech1}
    Let $\Lambda>0$ and $A_\Lambda\subset\partial E$ be  such that for all $x\in A_\Lambda$ we have  
 \[
     \limsup_{r\to 0^+}\frac{\mu(B(x,r))}{r^n}\ge\Lambda.
 \] 
 Then 
 \[
 \mathcal{H}^{n-1}_{Euc}(A_\Lambda)\le \frac{C}{\Lambda^{(n-1)/n}}\, P(E,\Omega).
 \]
 \end{lemma}
 
 Note that in the above limes supremum condition, if we replace $B(x,r)$ with $B^d(x,r)$, then by
 Lemma~\ref{Ahlfors} we have $A_\Lambda$ to be empty whenever $\Lambda>C$.
 
 \begin{proof}
   Fix $\delta>0$; then by the condition imposed on $A_\Lambda$, for every $x\in A_\Lambda$ we can
 find $0<r_x<\delta/5$ such that 
 \[
 \mu(B(x,r_x))\ge (\Lambda-\delta)r_x^n.
 \]  
 
 Note also that if $x\in A_\Lambda$ and $r>0$ such that 
 \[
 \mu(B(x,r))\ge (\Lambda-\delta) r^n,
 \] 
 then for any
 $y\in\partial B(x,r)$ we have that
 \[
   (\Lambda-\delta)r^n\le \mu(B(x,r))=\mu(B(x,|x-y|))\le Cd(x,y)^n,
 \]
 and so we have $d(x,y)\ge C^{-1/n}\, (\Lambda-\delta)^{1/n}\, r$. It follows that
 \[
    B^d(x,\tfrac{(\Lambda-\delta)^{1/n}}{C^{1/n}}\, r)\subset B(x,r).
 \]
 
 The family $B^d(x,\tfrac{(\Lambda-\delta)^{1/n}}{5C^{1/n}}\, r_x)$, $x\in A_\Lambda$,
 forms a cover of $A_\Lambda$, and hence we can find a pairwise disjoint countable subfamily
 \[
 \{B^d(x_i,\tfrac{(\Lambda-\delta)^{1/n}}{5C^{1/n}}\, r_i)\}_i
 \]  
 such that 
 \[
 A_\Lambda\subset\bigcup_i B^d(x_i,\tfrac{(\Lambda-\delta)^{1/n}}{C^{1/n}} r_i). 
 \]
 Hence by Lemma~\ref{Ahlfors}
 and the fact that the balls $\{B(x_i,r_i)\}_i$ therefore also form a cover of $A_\Lambda$ by Euclidean balls,  
 \begin{align*}
   \mathcal{H}_{Euc,\delta}^{n-1}(A_\Lambda)
        &\le \sum_i (r_i)^{n-1}  = \sum_i\frac{r_i^n}{r_i}\\
       & \le \frac{C}{\Lambda-\delta}\sum_i\frac{\mu(B^d(x_i,\tfrac{(\Lambda-\delta)^{1/n}}{C^{1/n}}r_i))}{r_i}\\
       &\le \frac{C}{\Lambda-\delta} \sum_i\frac{\mu(B^d(x_i,\tfrac{(\Lambda-\delta)^{1/n}}{5C^{1/n}} r_i))}{r_i}\\
       &\le \frac{C}{(\Lambda-\delta)^{1-\tfrac1n}}
               \sum_i\frac{\mu(B^d(x_i,\tfrac{(\Lambda-\delta)^{1/n}}{5C^{1/n}} r_i))}{\tfrac{(\Lambda-\delta)^{1/n}}{5C^{1/n}} r_i}.
 \end{align*}
By Lemma~\ref{upperGrowth}, we now have
\[
   \mathcal{H}_{Euc,\delta}^{n-1}(A_\Lambda)\le 
         \frac{C}{(\Lambda-\delta)^{1-\tfrac1n}}\, \sum_i P\left(E,B^d(x_i,\tfrac{(\Lambda+\delta)^{1/n}}{5C^{1/n}}\, r_i)\right).
\]
Since the family $\{B^d(x_i,\tfrac{(\Lambda+\delta)^{1/n}}{5C^{1/n}}\, r_i)\}_i$ is pairwise disjoint, we see that
\[
   \mathcal{H}_{Euc,\delta}^{n-1}(A_\Lambda)\le\frac{C}{(\Lambda-\delta)^{1-\tfrac1n}}\, P(E,\Omega).
\]
Letting $\delta\to 0$ completes the proof.
 \end{proof}
 
 \begin{lemma}\label{finite-density}
   For $\mathcal{H}^{n-1}_{Euc}$-almost every $x\in\partial E$ we have
   \[
      \limsup_{r\to 0}\frac{\mu(B(x,r))}{r^n}<\infty.
   \]
 \end{lemma}
 
 \begin{proof}
  By Lemma~\ref{Tech1}, we know that 
  \[
  \mathcal{H}^{n-1}_{Euc}(A_\Lambda)\le C\, \Lambda^{-(n-1)/n}P(E,\Omega).
  \]
  Since the set of all points $x\in \partial E$ for which 
  \[
  \limsup_{r\to 0}\frac{\mu(B(x,r))}{r^n}=\infty
  \] 
  is the set $\bigcap_{\Lambda>0}A_\Lambda$, we see that the claim of the lemma holds true.
 \end{proof}
 
 We set $F_0$ to be the collection of all points $x\in\partial E$ for which
 \[
     \limsup_{r\to 0^+}\frac{\mu(B(x,r))}{r^n}=0.
 \]
 Let $Z=\bigcap_{\Lambda>0}A_\Lambda$; from the above discussion we know that $\mathcal{H}_{Euc}^{n-1}(Z)$
 is zero.
 
 \begin{lemma}\label{partition}
   We have that $F_\infty=\bigcup_{k\in\N}(A_{1/k}\setminus Z)$ is $\sigma$-finite with respect to the 
   measure $\mathcal{H}^{n-1}_{Euc}$, and
 \[
    \partial E = Z\cup F_0\cup F_\infty.
 \]
 \end{lemma}
 
 \begin{proof}
  By Lemma~\ref{Tech1}, we know that
  $\mathcal{H}^{n-1}_{Euc}(A_{1/k})<\infty$. Thus we see that $F_\infty$ is $\sigma$-finite
  with respect to the measure $\mathcal{H}^{n-1}_{Euc}$.
 \end{proof}
 
 \begin{lemma}\label{Tech2}
  Either $\mathcal{H}(F_0)=0$ or $\mathcal{H}^{n-1}_{Euc}(F_0)=\infty$. Furthermore, 
  with 
  \[
  K_\eps=\left\{ x\in\partial E\, :\, \limsup_{r\to 0^+}\frac{\mu(B(x,r))}{r^n}<\eps\right\},
  \]
  we have
  \[
    \mathcal{H}\left(K_\eps\right)  \le C\, \eps^{(n-1)/n}\, \mathcal{H}^{n-1}_{Euc}\left(K_\eps\right).
  \]
 \end{lemma}
 
 Note that here $\mathcal{H}$ is the codimension $1$ Hausdorff measure with respect to the metric $d$ and measure $\mu$,
 while $\mathcal{H}^{n-1}_{Euc}$ is the $(n-1)$-dimensional Hausdorff measure with respect to the Lebesgue measure
 and Euclidean metric.
 
 \begin{proof}
   Suppose $\mathcal{H}(F_0)>0$. We will show that then $\mathcal{H}^{n-1}_{Euc}(F_0)=\infty$. To this end, 
 fix $\eps>0$.  For each $x\in F_0$ there is a positive number $\delta_x$ such that whenever $0<r<\delta_x$
 we have $\mu(B(x,r))\le 2\eps r^n$. For each $j\in\N$ let 
 \[
     F_j=\{x\in F_0\, :\, \delta_x\ge 1/j\}.
 \]
 Note that for large $j$ we have $F_j$ non-empty since $F_0=\bigcup_j F_j$ is non-empty. 
 Furthermore, because $\mathcal{H}(F_0)>0$ we have that $\mathcal{H}(F_j)>0$ for sufficiently large $j$.
 Let $0<\delta<1/j$, and for
 each $x\in F_j$, whenever $r<\delta$, we have that $\mu(B(x,r))\le 2\eps r^n$. For $y\in\partial B(x,r)$, we see by
 \eqref{eq:1A} that
 \[
   2\eps\, r^n\ge \mu(B(x,r))=\mu(B(x,|x-y|))\ge \frac{1}{C}\, d(x,y)^n.
 \]
 It follows that $d(x,y)\le C\eps^{1/n}\, r$ whenever $y\in\partial B(x,r)$. Therefore 
 $B(x,r)\subset B^d(x,C\eps^{1/n}r)$. Now we choose a countable cover of $F_j$ by balls $B(x_i,r_i)$ with
 $r_i<\delta$ and 
 \[
   \mathcal{H}^{n-1}_{Euc}(F_j)+\delta\ge \sum_i r_i^{n-1}.
 \]
 But then by Lemma~\ref{Ahlfors} we have 
 \begin{align*}
    \mathcal{H}^{n-1}_{Euc}(F_j)+\delta &\ge \frac{1}{C\eps}\sum_i\frac{\mu(B^d(x_i,C\eps^{1/n}r_i))}{r_i}\\
            &\ge \frac{1}{C\eps^{1-\tfrac{1}{n}}}\, \sum_i\frac{\mu(B^d(x_i,C\eps^{1/n}r_i))}{C\eps^{1/n}r_i}\\
            &\ge \frac{1}{C\eps^{1-\tfrac{1}{n}}}\, \mathcal{H}_{C\eps^{1/n}\delta}(F_j).
 \end{align*}
 Now letting $\delta\to 0$, we can conclude that 
 \[
    \mathcal{H}^{n-1}_{Euc}(F_0)\ge \mathcal{H}^{n-1}_{Euc}(F_j)\ge \frac{1}{C\eps^{1-\tfrac{1}{n}}}\, \mathcal{H}(F_j).
 \]
 Because $E$ is of finite perimeter and $E$ satisfies the density conditions discussed in the previous
 sections, we know that $\mathcal{H}(F_0)<\infty$. Also, if $j_1>j_2$ then $F_{j_1}\supset F_{j_2}$. Therefore we have
 $\mathcal{H}(F_0)=\lim_j\mathcal{H}(F_j)$. Therefore
 \[
       \mathcal{H}^{n-1}_{Euc}(F_0)\ge \frac{1}{C\eps^{(n-1)/n}}\, \mathcal{H}(F_0),
 \] 
 and now the desired conclusion that $\mathcal{H}^{n-1}_{Euc}(F_0)=\infty$ follows by taking $\eps\to 0$.
 The proof of the second claim of the lemma follows from an argument similar to the first part of the proof above.
 \end{proof}

Recall that 
\[
F_0:=\left\{x\in \partial E: \limsup_{r\rightarrow 0} \tfrac{\mu(B(x,r))}{r^n}=0\right\}.
\]

\begin{theorem}\label{main_rect_Euc}
With
$$
D_{\infty}:=\left\{x\in \partial E: \limsup_{r\rightarrow 0} \frac{\mathcal{H}_{Euc}^{n-1}(\partial E\cap B(x,r))}{r^{n-1}}=\infty\right\},
$$
the set $\partial E\backslash (F_0\cup D_{\infty})$ is $(n-1)$-rectifiable.    
 \end{theorem}

\begin{proof} 
Suppose not. Then combining Lemma~\ref{Tech1} with~\cite[page~205, Theorem~15.6]{Mat}, we know that 
there is  a purely $(n-1)$-unrectifiable set 
$K\subset  \partial E\backslash (F_0\cup D_{\infty})$ with $\mathcal{H}^{n-1}_{Euc}(K)>0$. Since $K\cap D_\infty$
is empty, it follows that $\mathcal{H}^{n-1}_{Euc}(K\cap B)$ is finite for balls $B$ centered at points in $K$ and with
sufficiently small radii. Thus there is a density point 
$x_0$ of $K$ for the measure $\mathcal{H}^{n-1}_{Euc}$.  By 
Lemma ~\ref{finite-density} we can assume without loss of generality that 
$$
Q:=\limsup_{r\rightarrow 0} \frac{\mu(B(x_0,r))}{r^n} <\infty.  
$$
Since $x_0\not\in F_0$, 
$$
\infty> Q=\limsup_{r\rightarrow 0} \frac{\mu(B(x_0,r))}{r^n} >0.
$$
Furthermore, because $x_0\not\in D_\infty$, we have 
\begin{equation*}
    M=\limsup_{r\rightarrow 0} \frac{\mathcal{H}_{Euc}^{n-1}(\partial E\cap B(x_0,r))}{r^{n-1}}<\infty.
\end{equation*}

Let $\epsilon$ be some small number to be determined later; by the choice of $x_0$, for sufficiently small $r_0>0$,
for all $0<r<r_0$ by the definition of $M$ above we have 
\begin{equation}\label{eq:extra2}
   \frac{\mathcal{H}^{n-1}_{Euc}((\partial E\backslash K)\cap B(x_0,r))}
       {\mathcal{H}^{n-1}_{Euc}(\partial E\cap B(x_0,r))}<\frac{\eps}{2M}.
\end{equation}
Since $x_0\not\in D_\infty$, for sufficiently small $r$ we also have 
\[
   \mathcal{H}^{n-1}_{Euc}(\partial E\cap B(x_0,r))\le (M+\eps) r^{n-1}.
\]
Therefore, by~\eqref{eq:extra2}, for sufficiently small $r>0$,
\begin{equation}
\label{eqsd1}
  \mathcal{H}^{n-1}_{Euc}((\partial E\backslash K)\cap B(x_0,r))<\epsilon r^{n-1}.
\end{equation}
We can find a small positive number $r>0$ that satisfies the above requirements and in addition satisfies
\[
   \mu(B(x_0,r))\geq \frac{Q}{2} r^n.
\]

By inequality~\eqref{eq:1A},
$$
B^d(x_0,C^{-1}\mu(B(x_0,r))^{\frac{1}{n}})\subset B(x_0,r)\subset B^d(x_0,C\mu(B(x_0,r))^{\frac{1}{n}}).
$$
So in particular, by the definition of $Q$ and the choice of $r$, we have that
$B^d(x_0,c Q^{\frac{1}{n}}r)\subset B(x_0,r)$ for $c=C^{-1}\, 2^{-1/n}$. 

Now by Theorem~\ref{porous1}, we can find $y_0\in E$ and $y_1\in E^c$ such that 
\[B^d(y_0,c_1 Q^{\frac{1}{n}}r)\subset E\cap B^d(x_0,c Q^{\frac{1}{n}}r) \subset B(x_0,r)\] and 
\[B^d(y_1, c_1 Q^{\frac{1}{n}}r)\subset B^d(x_0,c Q^{\frac{1}{n}}r)\setminus E\subset B(x_0,r)\].

For $i=0,1$ let 
$$
\gamma_i=\inf\left\{h>0: B^d(y_i, c_1 Q^{\frac{1}{n}}r)\subset B(y_i,h)\right\},
$$
that is, $B(y_i,\gamma_i)$ is the smallest Euclidean ball containing the metric ball $B^d(y_i,c_1 Q^{1/n}r)$.
Note that by Lemma~\ref{Ahlfors},
\[
    \mu(B(y_i,\gamma_i))\geq \mu(B^d(y_i,c_1 Q^{1/n}r))\ge c Q\, r^n.
\]
Because $B^d(y_i,c_1 Q^{1/n}r)\subset B(x_0,r)$, it follows that $B^d(y_i,c_1 Q^{1/n}r)\subset B(y_i,2r)$, and so
$\gamma_i\le 2r$.

Since $B(x_0,r)\subset B(y_i,2r)$ we have by the doubling property of $\mu$, with respect to the Euclidean metric,
that $\mu(B(y_i,2r))\leq CQr^n$.  

As observed above, $\gamma_i\le 2r$, and so~\eqref{UMB1} applies here to give
$$
\frac{cQr^n}{CQr^n}\leq \frac{\mu(B(y_i,\gamma_i))}{\mu(B(y_i,2r))}\leq C \left(\frac{\gamma_i}{2 r}\right)^{Q_1}.
$$ 
Thus $\gamma_i\geq c\, r$ and so by Lemma~\ref{App3} we have that 
$$
B(y_i,c\, r)\subset B^d(y_i, c_1 Q^{\frac{1}{n}}r).
$$

As $K$ is purely unrectifiable, by the Besicovich-Federer Projection theorem (\cite[Theorem 18.1(2)]{Mat}), for
$i=0,1$
there must exist points $\ti{y}_i\in  B(y_i,c\,r/4) $ such that for $v=\frac{\ti{y}_1-\ti{y}_2}{\left|\ti{y}_1-\ti{y}_2\right|}$ we have 
$$
\mathcal{H}^{n-1}_{Euc}(P_{v^{\perp}}(K))=0.
$$ 
Here $P_{v^\perp}$ is the projection to the $(n-1)$-dimensional hyperplane orthogonal to the vector $v$.

Let $\xi_1=v$; then we can find unit vectors   
$\xi_2,\xi_3,\dots \xi_n$  
such that $\left\{\xi_1,\xi_2,\dots \xi_n\right\}$ 
forms an  orthonormal basis for the vector space $\R^n$. For any $z\in \R^n$ and $\beta>0$ let $Q_{\beta}(z)$ denote the cube whose faces are normal to the vectors $\left\{\xi_1,\xi_2,\dots \xi_n\right\}$, with Euclidean side length $\beta$, and
center located at $z$. Note 
$$
Q_{\tfrac{c\, r}{8\, n^{1/2}}}(\ti{y}_0)\subset E
\quad\text{and}\quad
Q_{\tfrac{c\, r}{8\, n^{1/2}}}(\ti{y}_1)\subset \R^n\setminus E.
$$
Consider the following cross-section of the cube $Q_{\tfrac{c\, r}{8\, n^{1/2}}}(\ti{y}_0)$,
$$
\Pi:=Q_{\tfrac{c\, r}{8\, n^{1/2}}}(\ti{y}_0)\cap \left( \langle v\rangle^{\perp}+\ti{y}_0\right),  
$$
where $\langle v\rangle$ is the one-dimensional vector subspace of $\R^n$ spanned by the vector $v$, and
$\langle v\rangle^{\perp}$ is the $(n-1)$-dimensional hyperplane orthogonal to the vector $v$. For each $z\in \Pi$ let 
$$
e_z= P_{v^{\perp}}^{-1}(\{P_{v^{\perp}}(z)\})\cap 
Q_{\tfrac{c\, r}{8\, n^{1/2}}}(\ti{y}_1)\cap \left(\langle v\rangle^{\perp}+\ti{y}_1\right),
$$
that is, $e_z$ is the point in the region 
\[
Q_{\tfrac{c r}{8\, n^{1/2}}}(\ti{y}_1)\cap \left(\langle v\rangle^{\perp}+\ti{y}_1\right)=\Pi+\ti{y}_1-\ti{y}_0
\] 
corresponding to
$z\in\Pi$ such that $z-e_z=\ti{y_0}-\ti{y_1}$.

Let 
$$
\Pi'=\left\{z\in \Pi:P_{v^{\perp}}^{-1}(P_{v^{\perp}}(z))\cap K=\emptyset\right\},
$$
that is, $\Pi^\prime$ is the collection of all points $z\in\Pi$ such that the line segment $[z,e_z]$
connecting $z$ to $e_z$ does \emph{not}
intersect the purely unrectifiable set $K$. By the choice of $\ti{y}_i$, $i=0,1$, we know that
$\mathcal{H}^{n-1}_{Euc}(\Pi\setminus \Pi')=0$. 
However, for any $z\in \Pi'$ we know that $z\in E$ and $e_z\in E^c$ so 
we must have the line segment $\left[z,e_z\right]$ intersect $\partial E\backslash K$. So for 
each $z\in \Pi'$ we can pick a point $b_z\in   (\partial E\backslash K)\cap [z,e_z]$. 
On the other hand, as orthogonal projections do not increase the measure $\mathcal{H}^{n-1}_{Euc}$,
\begin{align*}
\mathcal{H}^{n-1}_{Euc}(\partial E\setminus K)&\ge \mathcal{H}^{n-1}_{Euc}\left(\bigcup_{z\in \Pi'} b_z\right) 
     \geq \mathcal{H}^{n-1}_{Euc}\left(P_{v^{\perp}}\left( \bigcup_{z\in \Pi'} b_z \right)\right)\\
    & =\mathcal{H}^{n-1}_{Euc}(\Pi^\prime)
    = \left(\frac{c\, r}{8\, \sqrt{n}}\right)^{n-1}= \left(\frac{c}{8\, \sqrt{n}}\right)^{n-1}\, r^{n-1}.
\end{align*}
This contradicts~\eqref{eqsd1} when we choose $0<\eps<\left(\frac{c}{8\, \sqrt{n}}\right)^{n-1}$.

\end{proof}

 \begin{corollary}
   Suppose that there is a positive number $\alpha$ such that $\omega(x)\ge\alpha$ for $\mathcal{L}^n$-almost
 every $x$ in a neighborhood of $\overline{E}$. Then $\partial E$ has $\mathcal{H}^{n-1}_{Euc}$ finite measure and is Euclidean $(n-1)$-rectifiable.
 \end{corollary}
 
 \begin{proof}
By Lemma~\ref{partition}, we have $\partial E=Z\cup F_\infty\cup F_0$, with $\mathcal{H}^{n-1}_{Euc}(Z)=0$
 and $\mathcal{H}^{n-1}_{Euc}$ being $\sigma$-finite on $F_\infty$.  It follows 
 from the assumption $\omega\ge\alpha$ almost everywhere  
 that for all $x\in\partial E$ we have
 \[
    \liminf_{r\to 0^+}\frac{\mu(B(x,r))}{r^n}\ge C_n\,\alpha>0,
 \]
 that is, $F_0$ is empty.  
 
 Now we look at $F_\infty=\bigcup_{k\in\N}A_{1/k}\setminus Z$. 
 Because of the assumption that $\omega\ge \alpha$ almost everywhere, we know that $F_\infty=A_{1/k_0}\setminus Z$
 where $k_0\in\N$ large enough so that $1/k_0<C_n\,\alpha$. So by Lemma \ref{Tech1} we have that 
\[
H^{n-1}(\partial E)=\mathcal{H}^{n-1}_{Euc}(F_{\infty})<\infty.
\] 
Thus an application of~\cite[Theorem~6.2]{Mat} gives
$\mathcal{H}^{n-1}_{Euc}(D_{\infty})=0$, and 
so $\partial E$ is rectifiable by Theorem~\ref{main_rect_Euc}.  
 \end{proof}

\noindent Address:\\
\noindent J.K.:  Aalto University, Department of Mathematics, P.O. Box 11100, 
FI-00076 Aalto, Finland. \\
\noindent 
E-mail: {\tt juha.kinnunen@tkk.fi}\\

\noindent R.K.:  Department of Mathematics and Statistics, 
P.O. Box 68 (Gustaf H\"allstr\"omin katu 2b), 
FI-00014 University of Helsinki, Finland. \\
\noindent 
E-mail: {\tt riikka.korte@helsinki.fi}\\

\noindent A.L.: Department of Mathematical Sciences, P.O. Box 210025, University of 
Cincinnati, Cincinnati, OH 45221--0025, U.S.A. \\
\noindent 
E-mail: {\tt andrew.lorent@uc.edu}\\

\noindent N.S.: Department of Mathematical Sciences, P.O. Box 210025, University of 
Cincinnati, Cincinnati, OH 45221--0025, U.S.A. \\
\noindent 
E-mail: {\tt nages@math.uc.edu} \\

\end{document}